\documentclass[11pt,a4paper,twoside]{amsart}
\usepackage[activeacute,english]{babel}
\usepackage[latin1]{inputenc}
\usepackage{amsfonts}
\usepackage{amsmath,amsthm}
\usepackage{amssymb}
\usepackage{graphics}

\newcommand{\CC}{{\mathcal C}}
\newcommand{\N}{{\mathbb N}}

\newcommand{\C}{{\mathbb C}}

\newcommand{\inte}{{\text{int}}}
\newcommand{\dime}{{\text{dim}}}

\newtheorem{teo}{Theorem}[section]
\newtheorem{lema}[teo]{Lemma}

\newtheorem{ques}[teo]{Question}

\title{Failure of rational approximation on some Cantor type sets}
\author{ALBERT MAS-BLESA}
\date{February 2008}
\thanks{Supported by grant AP2006-02416 (Programa FPU del MEC, España). Also, partially supported by grants 2005SGR-007749 (Generalitat de Catalunya) and MTM2007-62817 (MEC, España)}
\address{Departament de Matem\`atiques, Universitat Aut\`onoma de Bar\-ce\-lo\-na, Spain}
\email{amblesa@mat.uab.cat}

\begin{document}

\begin{abstract}
Let $A(K)$ be the algebra of continuous functions on a compact set $K\subset\C$ which are analytic on the interior of $K$, and $R(K)$ the closure (with the uniform convergence on $K$) of the functions that are analytic on a neighborhood of $K$. A counterexample of a question made by A. O'Farrell about the equality of the algebras $R(K)$ and $A(K)$ when $K=(K_{1}\times[0,1])\cup([0,1]\times K_{2})\subseteq\C$, with $K_{1}$ and $K_{2}$ compact subsets of $[0,1]$, is given. Also, the equality is proved with the assumption that $K_{1}$ has no interior.
\end{abstract}

\maketitle

\section{Introduction}
Consider a compact set $K$ of the complex plane. Let $A(K)$ be the algebra of continuous functions on $K$ which are analytic on the interior of $K$, and $R(K)$ the closure (with the uniform convergence on $K$) of the functions that are analytic on a neighborhood of $K$. Obviously, $R(K)\subseteq A(K)$.

In the 60's, Vitushkin gave a description in analytic terms
of the compact sets $K$ for which $R(K)=A(K)$ (see \cite{Vitushkin}), but there is still no characterization of those compact sets in a geometric way. Nevertheless, there have been important advances in this area recently, as can be seen in the articles of Xavier Tolsa \cite{Tolsa1} and
\cite{Tolsa2} and the one of Guy David \cite{David}. In this direction, Anthony G. O'Farrell raised the following question (private communication):
\begin{ques}\label{ques}
Let $K_{1}$ and $K_{2}$ be two compact subsets of $[0,1]$ and define $K=(K_{1}\times[0,1])\cup([0,1]\times K_{2})\subseteq\C.$ Is it true that $R(K)=A(K)$?
\end{ques}
It is known that the identity holds if one of the compact sets $K_{1}$ or $K_{2}$ has no interior. For completeness, we include a proof of that fact at the end of the paper. However, it was not known whether the identity holds or not in general.  In this paper we provide an example of a compact set $K$ which gives a negative answer to the question. The set $K$ is constructed as follows:

Let $\mathcal{C}(1/3)$ be the ternary Cantor set on the interval $[0,1]$, i.e., $$\mathcal{C}(1/3)=\bigcap_{n=0}^{\infty}\bigcup_{j=1}^{2^{n}}I_{n}^{j},$$ where
$I_{0}^{1}=[0,1]$ and each $I_{n}^{j}$ is an interval of length $3^{-n}$ obtained by dividing the intervals of length $3^{-n+1}$ in three equal parts and excluding the central part. Call $z_{n}^{j}$ the center of $I_{n}^{j}$. Consider a sequence $\delta_{n}>0$ such that $\delta_{n}<3^{-n-1}$ and define $J_{n}^{j}=(z_{n}^{j}-\delta_{n}/2,z_{n}^{j}+\delta_{n}/2)$, where $z_{n}^{j}$ is the center of $I_{n}^{j}$. Let $$E_{m}=[0,1]\setminus\bigcup_{n=0}^{m}\bigcup_{j=1}^{2^{n}}J_{n}^{j}.$$
Finally, define $F_{m}=(E_{m}\times[0,1])\cup([0,1]\times E_{m}) \subseteq\mathbb{C}$ and put
$K=\bigcap_{m=0}^{\infty}F_{m}$.

With this construction of $K$ we will prove the main result of the paper:
\begin{teo}\label{main teo}
For a suitable choice of the sequence $\delta_{n}$, $R(K)\neq A(K).$
\end{teo}

In the whole paper $\mathcal{M}^{1}$ stands for the {\em1-dimensional Hausdorff content} and $\alpha$ denotes the {\em continuous analytic capacity} (see \cite{Vitushkin}). Remember that, given a compact set $F\subseteq\C$,
$$\alpha(F) = \sup|f'(\infty)|,$$
where the supremum is taken over all continuous functions
$f:\C\longrightarrow \C$ which are analytic on $\C\setminus F$, and uniformly bounded by $1$ on $\C$. If $f$ satisfies all these properties, we say that $f$ is {\em admissible} for $\alpha$ and $F$. By definition, $f'(\infty)=\lim_{z\rightarrow\infty}z(f(z)-f(\infty))$.

\begin{figure}[ht]
\begin{center}
\scalebox{1.6}{\includegraphics{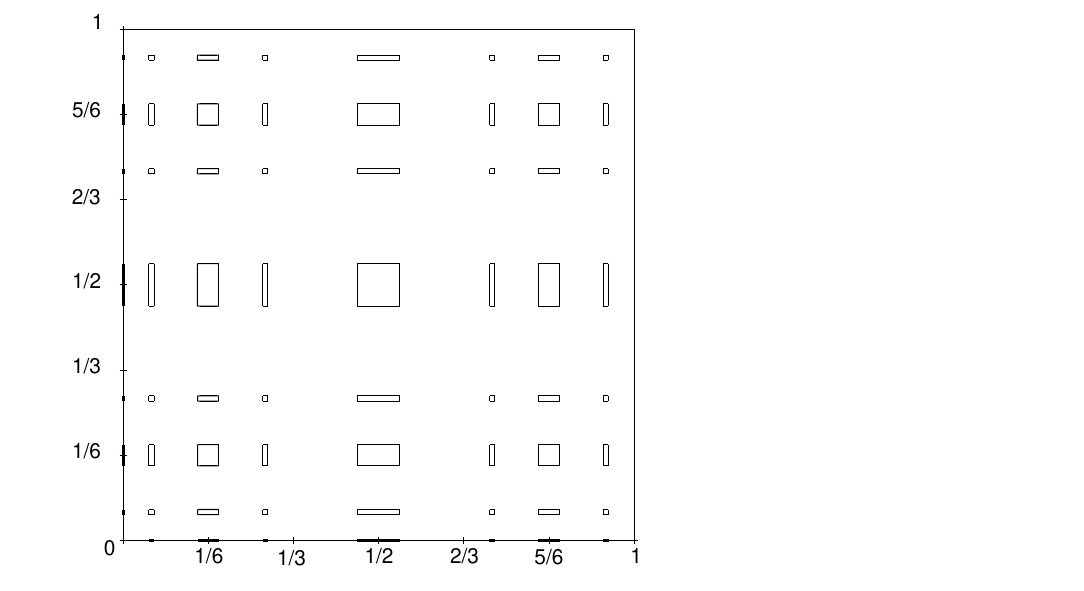}}
\caption{\label{fig1} This is a picture of the compact set $F_{2}$. The rectangles inside $[0,1]^{2}$ are the holes of $F_{2}$, and the bold lines on the sides of $[0,1]^{2}$ correspond to the subset of the real line $\bigcup_{n=0}^{2}\bigcup_{j=1}^{2^{n}}J_{n}^{j}$.}
\end{center}
\end{figure}

\section{Proof of the main result}
In the two following lemmas, we shall obtain some estimates of the Hausdorff content of $[0,1]^{2}\setminus K$ that will be useful to show that the algebras $R(K)$ and $A(K)$ are not equal for a suitable choice of the sequence $\delta_{n}$.

\begin{lema}
Fix $n_{0}\in\mathbb{N}$ and $\delta>0$ such that $\delta<3^{-n_{0}+2}$. Define
$\tilde{J}_{n}^{j}=(z_{n}^{j}-\delta/2,z_{n}^{j}+\delta/2)$, $R_{n}^{j}=\tilde{J}_{n}^{j}\times [0,\delta]$ and
$$R=\bigcup_{n=0}^{n_{0}}\bigcup_{j=1}^{2^{n}}R_{n}^{j}.$$
Then $\mathcal{M}^{1}(R)< 8\delta^{\eta}$, where $\eta=1-\frac{1}{\log_{2}3}>0$.
\end{lema}

\begin{proof}

Since $R$ is the union of the squares $R_{n}^{j}$ for $0\leq n\leq n_{0}$ and $1\leq j\leq2^{n}$ and each square has side length $\delta$, we have
\begin{eqnarray*}
\mathcal{M}^{1}(R)\leq\sum_{n=0}^{n_{0}}\sum_{j=1}^{2^{n}}\delta=\delta(2^{n_{0}+1}-1)\leq \delta2^{n_{0}+1}.
\end{eqnarray*}
The inequality $\delta<3^{-n_{0}+2}$ is equivalent to $n_{0}<2-\log_{3}{\delta}$. Then, using that
$\log_{3}{\delta}=\log_{2}{\delta}/\log_{2}{3}$, we can deduce that
$$\delta2^{n_{0}+1}<\delta2^{3-\log_{3}{\delta}}=\delta2^{3-\frac{\log_{2}{\delta}}{\log_{2}{3}}}=
8\delta^{1-\frac{1}{\log_{2}{3}}}=8\delta^{\eta}.$$
\end{proof}

As we will see in the proof of the following lemma, the important fact of the preceding one is that $\mathcal{M}^{1}(R)$ is bounded by something that tends to zero as $\delta$ decreases, rather than the exact value of the bound.

\begin{lema}
For every $\varepsilon>0$ there exists a sequence $\delta_{n}$ such that $$\mathcal{M}^{1}([0,1]^{2}\setminus K)< \varepsilon.$$
\end{lema}

\begin{proof}
Put $G=[0,1]^{2}\setminus K$. Consider the crosses $P_{n}^{k}$ for $k=1,\ldots,4^{n}$ defined in the following way (see Figure \ref{fig1} to understand the construction):
\vskip-2pt
\begin{align*}
&P_{0}^{1}=(J_{0}^{1}\times[0,1])\cup([0,1]\times J_{0}^{1}),\\
\text{ }\\
&P_{1}^{1}=(J_{1}^{1}\!\times\![0,1/3])\cup([0,1/3]\!\times\! J_{1}^{1}),\;
P_{1}^{2}=(J_{1}^{2}\!\times\![0,1/3])\cup([2/3,1]\!\times\! J_{1}^{1}),\\
&P_{1}^{3}=(J_{1}^{1}\!\times\![2/3,1])\cup([0,1/3]\!\times\! J_{1}^{2}),\;
P_{1}^{4}=(J_{1}^{2}\!\times\![2/3,1])\cup([2/3,1]\!\times\! J_{1}^{2}),\\
\text{ }\\
&P_{2}^{1}=(J_{2}^{1}\!\times\![0,1/9])\cup([0,1/9]\!\times\! J_{2}^{1}),\,
P_{2}^{2}=(J_{2}^{2}\!\times\![0,1/9])\cup([2/9,1/3]\!\times\! J_{2}^{1}),\\
&P_{2}^{3}=(J_{2}^{3}\!\times\![0,1/9])\cup([2/3,7/9]\!\times\! J_{2}^{1}),\,
P_{2}^{4}=(J_{2}^{4}\!\times\![0,1/9])\cup([8/9,1]\!\times\! J_{2}^{1}),\\
&P_{2}^{5}=(J_{2}^{1}\!\times\![2/9,1/3])\cup([0,1/9]\!\times\! J_{2}^{2}),\,\ldots
\end{align*}
\vskip10pt
It is clear that $G\subseteq\bigcup_{n=0}^{\infty}\bigcup_{k=1}^{4^{n}}P_{n}^{k}$. By construction, we also have
$\mathcal{M}^{1}(P_{n}^{1}\cap G)=\mathcal{M}^{1}(P_{n}^{k}\cap G)$ for all $k=1,\ldots,4^{n}$.
Therefore, $$\mathcal{M}^{1}(G)\leq \sum_{n=0}^{\infty}4^{n}\mathcal{M}^{1}(P_{n}^{1}\cap G).$$
Call $X_{n}$ the horitzontal strip of the cross $P_{n}^{1}$ and $Y_{n}$ the vertical one. Because of the symmetry of the compact set $K$ and the subadditivity of $\mathcal{M}^{1}$,
$$\mathcal{M}^{1}(P_{n}^{1}\cap G)\leq2\mathcal{M}^{1}(X_{n}\cap G).$$

Observe that $G$ is a countable union of rectangles, and on $X_{n}$ all those rectangles have the sides of length less or equal than $\delta_{n}$. So, the set $3^{n}(X_{n}\cap G):=\{3^{n}x:\;x\in X_{n}\cap G\}$ can be included by a translation in a set $R:=\bigcup_{n=0}^{n_{0}}\bigcup_{j=1}^{2^{n}}R_{n}^{j}$ like the one of the preceding lemma, if we take $\delta=3^{n}\delta_{n}$ and $n_{0}\in\mathbb{N}$ such that $3^{-n_{0}+1}\leq\delta<3^{-n_{0}+2}$. Applying the lemma we obtain,
$$\mathcal{M}^{1}(X_{n}\cap G)<3^{-n}8(3^{n}\delta_{n})^{\eta}=3^{n(\eta-1)}8\delta_{n}^{\eta}$$
with $\eta=1-\frac{1}{\log_{2}3}$, and then,
$$\mathcal{M}^{1}(G)\leq 8\sum_{n=0}^{\infty}4^{n}\mathcal{M}^{1}(X_{n}\cap G)< 8\sum_{n=0}^{\infty}4^{n}3^{n(\eta-1)}\delta_{n}^{\eta}.$$
Given $\varepsilon>0$, it is easy to find a decreasing sequence $\delta_{n}$ that makes the last sum less than $\varepsilon$, because $\eta>0$.
\end{proof}

\begin{proof}[Proof of theorem \ref{main teo}]
As Vitushkin proved in \cite{Vitushkin} (see also \cite{Gamelin-unif}, theorem VIII.8.2), $R(K)=A(K)$ if and only if  $\alpha(D\setminus K)=\alpha(D\setminus \inte K)$ for every bounded open set $D$.

If $\mathcal{C}=\mathcal{C}(1/3)\times\mathcal{C}(1/3)$, we known that $\alpha(\mathcal{C})>0$ because $\dime (\CC)>1$, where $\dime(\cdot)$ denotes the Hausdorff dimension. Observe that $\CC\subseteq\partial K$ and it does not depend on the chosen sequence $\delta_{n}$. This implies that $\alpha(\partial K)\geq\alpha(\CC)$, so it is guarantied a minimum of continuous analytic capacity on the boundary of $K$ for any sequence $\delta_{n}$.

Observe also that $\alpha([0,1]^{2}\setminus\inte K)=\alpha((0,1)^{2}\setminus\inte K)$  because $\partial([0,1]^{2})$ is {\em negligible} (see \cite{Gamelin-unif}, chapter VIII). Therefore,
$$\alpha(\CC)\leq\alpha(\partial K)\leq\alpha([0,1]^{2}\setminus\inte K)=\alpha((0,1)^{2}\setminus\inte K).$$

On the other hand, by the preceding lemma we can find a sequence $\delta_{n}$ such that
$\mathcal{M}^{1}([0,1]^{2}\setminus K)\leq\alpha(\CC)/2.$ If we take into account that $\alpha\leq\mathcal{M}^{1}$, we can deduce that
$$\alpha((0,1)^{2}\setminus\ K)\leq\mathcal{M}^{1}([0,1]^{2}\setminus K)\leq\alpha(\CC)/2
<\alpha(\CC)\leq\alpha((0,1)^{2}\setminus\inte K).$$
These inequalities show that the necessary condition for $R(K)=A(K)$ in Vitushkin's theorem does not hold for $D=(0,1)^{2}$. So, for that sequence $\delta_{n}$ we have $R(K)\neq A(K)$.
\end{proof}

\section{$A(K)=R(K)$ when $K_{1}$ has no interior}
Now, as we said at the beginning of the paper, we proceed to give an affirmative answer to the question \ref{ques} with the assumption that $K_{1}$ has no interior. We need an auxiliary lemma that we guess is already known, so we only sketch the proof.

\begin{lema}
Fix $\delta>0$ and $n\in\N$. Let $R$ be a rectangle with sides of length $\delta$ and $n\delta$ and put $R=\bigcup_{j=1}^{n}Q_{j}$, where $Q_{j}$ squares of side length $\delta$ with pairwise disjoint interiors. Let $E_{j}\subseteq Q_{j}$ and suppose there exists $C_{0}>0$ such that $\alpha(E_{j})\geq C_{0}\delta$ for all $j$. Then, there exists a constant $C_{1}>0$ depending only on $C_{0}$ such that
$$\sum_{j=1}^{n}\alpha(E_{j})\leq C_{1}\alpha(\bigcup_{j=1}^{n} E_{j}).$$
\end{lema}

\begin{proof}[Hint of the proof]
Given admissible functions $f_{j}$ for $\alpha$ and $E_{j}$, one can find a function $f$ admissible for $\alpha$ and $\bigcup_{j=1}^{n} E_{j}$ such that $\sum_{j}|f_{j}'(\infty)|=C_{1}|f'(\infty)|$ using {\em Vitushkin's localization scheme} with a modified {\em triple zero lemma} (see \cite{Verdera-nato} or \cite{Vitushkin}), where one uses the fact that the sets $E_{j}$ are aligned. Then, one can prove the lemma by taking supremums.
\end{proof}

From now on, we shall denote by $C$ an absolute constant that may change its value at different occurrences.

\begin{teo}\label{noint}
Let $K_{1},K_{2}\subseteq[0,1]$ be two compact sets and define $K=(K_{1}\times[0,1])\cup([0,1]\times K_{2})$. Suppose that $K_{1}$ has no interior. Then, $R(K)=A(K)$.
\end{teo}

\begin{proof}
By Vitushkin's theorem, it is known that $R(K)=A(K)$ if and only if there exists an absolute constant $C>0$ such that $\alpha(Q\setminus \inte K)\leq C\alpha(Q\setminus K)$ for all open squares $Q$.

Fix a square $Q$ of side length $l>0$. We can suppose that $Q\setminus K$ is not empty, so there exists a square $F\subseteq Q\setminus K$. Let $\pi_{x}$ and $\pi_{y}$ be the projections onto the horitzontal and vertical coordinate axis respectively. Then, $\pi_{y}(F)\subseteq\pi_{y}(Q)\setminus K_{2}$ and we can find an interval $F_{y}\subseteq \pi_{y}(F)$ of length $l/n$ for $n$ big enough.

On the other hand, if we split $\pi_{x}(Q)$ into intervals $I_{j}$ for $j=1,\ldots,n$ with pairwise disjoint interiors and length $l/n$, we can also find intervals $F_{x}^{j}\subseteq (\pi_{x}(Q)\setminus K_{1})\cap I_{j}$ for $j=1,\ldots,n$, because $K_{1}$ has no interior. Therefore, $\bigcup_{j=1}^{n}(F_{x}^{j}\times F_{y})\subseteq Q\setminus K$ and
$\alpha(F_{x}^{j}\times F_{y})\geq C_{0}\,l/n$.

Now we are ready to use the preceding lemma with the squares $Q_{j}=F_{y}\times I_{j}$, the subsets $E_{j}=F_{x}^{j}\times F_{y}$ and $\delta=l/n$, and we obtain
$$l\leq C\sum_{j=1}^{n}\alpha(F_{x}^{j}\times F_{y})\leq C\alpha(\bigcup_{j=1}^{n}(F_{x}^{j}\times F_{y}))\leq C\alpha(Q\setminus K).$$
We can finally deduce that
$$\alpha(Q\setminus \inte K)\leq\alpha(Q)=C\,l\leq C\alpha(Q\setminus K)$$
for every open square $Q$, so $R(K)=A(K)$.
\end{proof}

We are grateful to Anthony O'Farrell for the communication of another proof of theorem \ref{noint} which uses annihilating measures instead of Vitushkin's theorem.

\section*{Acknowledgment}
The author gratefully acknowledges Mark Melnikov and Xavier Tolsa for the communication of the main question \ref{ques} and for useful discussions while preparing this paper.

\end{document}